\documentclass[10pt]{amsart}

\usepackage{times,amsfonts,amsmath,amstext,amsbsy,amssymb,
  amsopn,amsthm,upref,eucal, amscd}
\usepackage[T1]{fontenc}
\usepackage[pdftex]{graphicx}

\newtheorem{theorem}{Theorem}[section]
\newtheorem{lemma}[theorem]{Lemma}

\newtheorem{proposition}[theorem]{Proposition}

\numberwithin{equation}{section}

\theoremstyle{definition}
\newtheorem{definition}[theorem]{Definition}

\newtheorem{remark}[theorem]{Remark}

 \def\Rset{\mathbb{R}}
 \def\Zset{\mathbb{Z}}
 \def\Tset{\mathbb{T}}

\def\wh{\widehat}

\def\leq{\leqslant }
\def\geq{\geqslant}
\def\RA{\mathbb{R}^{\mathcal A}}
\def\A{\mathcal{A}}
\def\iat{I_{\alpha}^t}
\def\iab{I_{\alpha}^b}

\def\RS{\mathbb{R}^{\Sigma}}
\def\RSO{\mathbb{R}_0^{\Sigma}}

\def\cS{{\mathcal S}}

\def\CruT{C^r({\bf u}(T))}

\begin{document}

\title[H\"older regularity of the solutions of the cohomological equation for Roth type i.e.m.]{H\"older regularity of the solutions of the cohomological equation for Roth type interval exchange maps}

\author[S. Marmi]{Stefano Marmi}
\address{Scuola Normale
Superiore and  C.N.R.S. UMI 3483 Laboratorio Fibonacci, Piazza dei Cavalieri 7, 56126 Pisa, Italy}
\email{ s.marmi(at)sns.it}
\author[J.-C. Yoccoz]{Jean-Christophe Yoccoz}
\address{Coll\`ege de France, 3, Rue d'Ulm, 75005
Paris, France}
\email{ jean-c.yoccoz(at)college-de-france.fr }

\date{July 7, 2014}


\begin{abstract}
We prove that the solutions of the cohomological equation for Roth type interval exchange maps 
are H\"older continuous provided that the datum is of class $C^r$ with $r>1$ and belongs to a finite-codimension
linear subspace. 
\end{abstract}

\maketitle
%\pageheight{8.5truein}
%\pagewidth{6.5truein}

\tableofcontents

%%%%%%%%%%%%%%%%%%%%%%%%%%%%%
%%%%%%%%%%%%%%%%%%%%%%%%%%%%%%
\section{Introduction}\label{Intro}
%%%%%%%%%%%%%%%%%%%%%%%%%%%%
%%%%%%%%%%%%%%%%%%%%%%%%%%%

\subsection{Roth type interval exchange maps}

An {\em interval exchange map} (i.e.m.)  $T$ on a (finite length) interval $I$ is
a one-to-one map which is locally a translation except at a finite number of discontinuities.
Clearly $T$ is orientation-preserving and preserves Lebesgue measure. 
When the number $d$ of intervals of continuity of $T$  is equal to $2$ 
then (modulo identification of
the endpoints of $I$) an i.e.m. corresponds to a rotation of the circle.
It can be thought as the first return map
of a linear flow on a two--dimensional torus on a transversal
circle. 
Analogously when $d\ge 3$ by singular suspension any
i.e.m.\ is related to the linear flow on a suitable translation
surface (see, e.g.\ [Ve1] for details, or section 2.2 below)
typically having genus higher than $1$. 
All translation surfaces obtained by suspension from i.e.m.\ with the same combinatorics have the same genus $g$,
and the same number $s$ of marked points; these numbers are related to the number $d$ of intervals of continuity of $T$
by the formula $d=2g+s-1$.

Rauzy and Veech have defined an  algorithm
that generalizes the classical continued fraction
algorithm (corresponding to the choice $d=2$) and associates to an i.e.m.\ another i.e.m.\ which is its first
return map to an appropriate subinterval [Ra, Ve2]. 
The Rauzy-Veech algorithm stops if and only if the i.e.m.\ has a {\it
connection},
i.e.\ a finite orbit which starts and ends at a discontinuity.
Both
 Rauzy--Veech "continued fraction" algorithm and its accelerated
version due to Zorich [Zo2] (see also [MMY1]) are ergodic w.r.t.\ an absolutely continous
invariant measure in the space of normalized standard i.e.m.'s. 
The ergodic properties of these renormalization dynamics in parameter space have been studied in detail ([Ve3],[ Ve4], [Zo3], [Zo4], [AGY], [B], [AB],  [Y4]).

In \cite{MMY1} we introduced a class of interval exchange maps, called {\it Roth type}, which has full measure in parameter space and for which 
the cohomological equation could be solved (after a finite-dimensional correction) 
with a loss of differentiability of about two derivatives (one obtains a continuous solution starting from a datum 
of class $C^1$ on each interval of continuity of $T$ and with a derivative of bounded variation). Roth type i.e.m.\ generalize Roth type irrational numbers. 
Roth type numbers have a purely arithmetical definition but can also be characterized by means of the regularity of the solutions of the  
cohomological equation associated to the rotation $R_\alpha\, :\,
x\mapsto x+\alpha$ on the circle $\Tset=\Rset/\Zset$. Indeed an irrational number $\alpha$ is of Roth
type if and only if for all $r,s\in\Rset$ with $r>s+1\ge 1$ and for
all functions $\Phi$ of class $C^r$ on $\Tset$ with mean zero,
there exists a unique function $\Psi$ of class
$C^s$ on $\Tset$ with mean zero such that $\Phi=\Psi\circ
R_\alpha-\Psi$.

The goal of this paper is to prove Theorem \ref{thmHolder1},  i.e.\ a stronger regularity result  than in \cite{MMY1} for the solutions of the cohomological equation and which is closer to the optimal result (in the case of circle rotations) quoted above.  
It is stated in full generality in section 3.4. Here we give a less precise statement:

\smallskip
\begin{theorem}
Let $T$ be an interval
exchange map with no connection and of restricted Roth type. 
Let $r $ be a real number $>1$. There exists $\bar\delta >0$ such that given any function 
$\varphi$ of class  $C^r$ on each interval of continuity of $T$ which belongs to the kernel of the boundary operator (defined in section 2.6) one can 
find a piecewise constant function $\chi$ and  a $\bar\delta$--H\"older continuous function $\psi $ such that
$$ \varphi = \chi + \psi \circ T - \psi\; . $$
\end{theorem}
\smallskip\noindent

In the theorem one has $\bar\delta < r-1$ and 
the i.e.m.\ $T$ are the Roth type i.e.m.\ for which the Lyapunov exponents of the
KZ-cocycle (see section 2.5) are non zero (in \cite{MMY2} we call this {\it restricted Roth type}). They still form a full measure set
by Forni's theorem \cite{For2}.

The first step in the direction of extending small divisor results
beyond the torus case was achieved by Forni's important paper \cite{For1}
on the cohomological equation associated to linear flows on
surfaces of higher genus. He solved
the cohomological equation on any translation surface for almost every direction and worked in the Sobolev scale obtaining a 
loss of differentiability of at most  $3+\varepsilon $
derivatives (for every $\varepsilon  >0$). 
The improved loss of regularity obtained in \cite{MMY1} (w.r.t.\ \cite{For1}) turned out to be decisive for the proof of a local conjugacy theorem for 
deformations of interval exchange maps (proven in \cite{MMY2}). We recall that the cohomological equation is the linearization of 
the conjugacy equation $T\circ h=h\circ T_0$ for a generalized i.e.m.\
$T$ (see \cite{MMY2}) close to the standard i.e.m.\ $T_0$.
In \cite{For3} sharper results on the loss of differentiability for solutions of the cohomological equations for translations surfaces have been proved: for
almost all translation surface in every stratum of the moduli space and for almost all direction the loss of Sobolev regularity is $1+\varepsilon$ (for any $\varepsilon >0$). As in \cite{For1} the i.e.m.\ case is not considered, nor an explicit diophantine condition is obtained.  Our Theorem \ref{thmHolder1} 
deals with the loss of H\"older regularity (which turns out to be $1+\varepsilon$ for {\it some} $\varepsilon>0$) and we get an explicit diophantine condition, but we do not know if almost every direction on a given translation surface leads  
to a restricted Roth type i.e.m.\ .

The tools of the proof of Theorem \ref{thmHolder1} allow us to prove Theorem \ref{thmcohom2}, a result of independent interest: when 
$T$ is of restricted Roth type, for a function $\varphi$ of class $C^1$ on each interval of continuity of $T$ which belongs to the kernel of the boundary 
operator, there exists a piecewise constant function $\chi$ such that the growth of the Birkhoff sums of $(\varphi - \chi )$ is slower that any positive power of
time. 

\subsection{Summary of the contents}

In section 2 we introduce
interval exchange maps and we develop the continued fraction
algorithms to an extent which will allows us to introduce Roth type
i.e.m.. 
Sections 2.1 and 2.2 are devoted to recalling the basic definitions and the construction of the suspension. 
The elementary step of the 
Rauzy-Veech continued fraction algorithm is described in section 2.3. When the i.e.m.\ has no connection
the Rauzy-Veech  algorithm (section 2.4) 
can be iterated indefinitely. In section 2.5 we introduce the discrete time
Kontsevich-Zorich cocycle and we discuss its relationship with return times 
and with  (special) Birkhoff sums of piecewise constant functions. After having introduced the boundary 
operator (section 2.6) and having briefly described its homological interpretation we conclude (section 2.7) 
with some consequences of the symplecticity of the Kontsevich-Zorich cocycle.

Section 3 is devoted to the study of the cohomological equation
and to the proof of our main theorem.
We first recall (section 3.1) the definition of i.e.m. of restricted Roth type which was given in \cite{MMY2} and deduce (section 3.2) some consequences of the 
hyperbolicity of the Kontsevich-Zorich cocycle. 
After reviewing previous results (section 3.3) on the regularity of the solutions of the cohomological equation \cite{For1,For3}, and especially
those obtained in \cite{MMY1} and refined in \cite{MMY2}, in section 3.4 we give a precise statement of our main result Theorem
\ref{thmHolder1}. 

The strategy of the proof is similar to our previous papers. In section 3.5 
we prove Theorem \ref{thmcohom2}
by estimating special Birkhoff sums of piecewise $C^1$ functions. 
In section 3.6 we obtain better estimates for piecewise $C^r$ functions. 
In sections 3.7 and 3.8 we introduce time and space decompositions which relate in particular general Birkhoff sums to special Birkhoff sums.
In section 3.9 we obtain the required H\"older estimate for pairs of points which are in special relative position; then the space decomposition
allows us to get the general estimate (section 3.10). Section 3.11 deals with the case of higher differentiability ($r>2$).  

\vskip 1. truecm \noindent {\bf Acknowledgements} This research has
been supported by the  following institutions: CNRS, MIUR (PRIN 2011 project ''Teorie geometriche e analitiche dei sistemi Hamiltoniani
in dimensioni finite e infinite''), the ANR project 851 GeoDyM, 
the Coll\`ege de France and
the Scuola Normale Superiore. We are also grateful to the two
former institutions and to the Centro di Ricerca Matematica
``Ennio De Giorgi'' in Pisa for hospitality.

%%%%%%%%%%%%%%%%%%%%%%%%%%%%%
%%%%%%%%%%%%%%%%%%%%%%%%%%%%%%
\section{Background on interval exchange maps}\label{Seciem}
%%%%%%%%%%%%%%%%%%%%%%%%%%%%%%
%%%%%%%%%%%%%%%%%%%%%%%%%%%%%%%

%%%%%%%%%%%%%%%%%%%%%%%%%%%%%%%%
\subsection{Interval exchange maps} \label{ssdefiem}
%%%%%%%%%%%%%%%%%%%%%%%%%%%%%

Let $\A$ be an alphabet with $d \geq 2$ symbols.

\begin{definition}\label{defiem1}
An {\it  interval exchange map} (i.e.m.\ ) $T$, acting upon an open bounded interval $I(T)$, is defined by 
two partitions mod.$0$ of $I(T)$  into $d$ open subintervals
indexed by $\A$ (the {\it top} and {\it bottom} partitions):
$$ I(T) = \sqcup \iat (T)= \sqcup\iab (T)\; \hbox{mod}\,  0$$
such that $|\iat (T)| =|\iab (T)|$ for all $\alpha\in\A$. The restriction of $T$ to 
$\iat (T)$ is the translation onto $\iab (T)$. 

The domain of the map $T$ is $\sqcup \iat (T)$, the domain of $T^{-1}$ is $ \sqcup\iab (T)$.
\end{definition}

 \begin{definition}\label{defiem3}
Let $T$ be an i.e.m. The points  $u_1(T)<\cdots<u_{d-1}(T)$ separating the $ \iat (T)$ are called the {\it singularities  of $T$}. The points
$v_1(T)<\cdots<v_{d-1}(T)$ separating the $ \iab (T)$ are called the {\it singularities  of $T^{-1}$}. We also write $I(T)=(u_0(T),u_d(T)) = (v_0(T),v_d(T))$.
 \end{definition} 
 
 \begin{definition}\label{defiem4}
The {\it combinatorial data} of an i.e.m $T$ is the pair $\pi = (\pi_t,\pi_b)$ of bijections from $\A$ onto $\{1,\ldots,d\}$ such
that
$$\iat (T)= (u_{\pi_t(\alpha)-1}(T),u_{\pi_t(\alpha)}(T)),\quad \iab (T)= (v_{\pi_b(\alpha)-1}(T),v_{\pi_b(\alpha)}(T))$$
for each $\alpha \in \A$.
\end{definition}

We always assume that the combinatorial data are {\it irreducible}: for $1 \leq k<d$, we have
$$ \pi_t^{-1}(\{1,\ldots ,k\})\not=\pi_b^{-1}(\{1,\ldots ,k\})\; . $$

%%%%%%%%%%%%%%%%%%%%%%%%%%%
\subsection{Suspension and genus}\label{ssSuspension}
%%%%%%%%%%%%%%%%%%%%%%%%%%%%

Let $T$ be an  i.e.m.\ with combinatorial data $\pi= (\pi_t,\pi_b)$ acting on an interval $I(T) = (u_0(T),u_d(T))$. For $\alpha \in \A$ define
$$\lambda_{\alpha} = |\iat (T)|=|\iab (T)|, \qquad \zeta_{\alpha}=\lambda_{\alpha}
+\sqrt {-1} \,(\pi_b(\alpha)-\pi_t(\alpha)) ,$$
and then, for $0 \leq i \leq d$
$$ U_i = u_0(T) + \sum_{\pi_t \alpha \leq i} \zeta_{\alpha}, \quad V_i = u_0(T) + \sum_{\pi_b \alpha \leq i} \zeta_{\alpha}.$$
One has $U_0 = V_0 =u_0(T)$, $U_d = V_d  = u_d(T)$, and $\Im U_i >0>\Im V_i$ for 
$0<i<d$. The $2d$ segments $[U_{i-1},U_i]$, $[V_{i-1},V_i]$ ($0<i\leq d$) form the boundary of a polygon. Gluing the parallel top and bottom $\zeta_{\alpha}$ sides of this polygon produces a translation surface $M_T$ ([Zo1]), in which the vertices of the polygon define a set $\Sigma$ of marked points. 
\par
The cardinality $s$ of $\Sigma$, the genus $g$ of $M_T$ and the number $d$ of intervals are related by
$$d=2g+s-1\;.$$
Both $g$ and $s$ only depend on $\pi$. 

%%%%%%%%%%%%%%%%%%%%%%%%%%%
%\subsection{Symplectic structure}\label{ssSymplectic}
%%%%%%%%%%%%%%%%%%%%%%%%%%%%

%There are several ways to compute $g$ and $s$ from $\pi$. We indicate one way in section \ref{Secboundary}.
The genus $g$ can be computed from the combinatorial data as follows. Define an antisymmetric matrix $\Omega =
\Omega (\pi)$ by
\begin{equation*}
\Omega_{\alpha \,\beta}=\left\{
\begin{array}{cc}
+1 & \text{if } \pi_t(\alpha)<\pi_t(\beta),\pi_b(\alpha)>\pi_b(\beta),\\
-1 & \text{if } \pi_t(\alpha)>\pi_t(\beta),\pi_b(\alpha)<\pi_b(\beta),\\
0 &\text{otherwise.}
\end{array} \right.
\end{equation*}
Then the rank of $\Omega$ is $2g$. Actually ([Y1],[Y4]), if one identifies $\RA$ with the relative homology group
$H_1(M_T,\Sigma,\Rset)$ via the basis defined by the sides $\zeta_{\alpha}$ of the polygon, the image of $\Omega$ coincides
with the absolute homology group $H_1(M_T,\Rset)$. Another way to compute $s$ (and thus $g$) consists in going around the
marked points, as explained in section 2.6.

%%%%%%%%%%%%%%%%%%%%%%%%%%%%%%%%%
\subsection{The elementary step of the Rauzy-Veech algorithm}\label{ssbasicRV}
%%%%%%%%%%%%%%%%%%%%%%%%%%%%%%%%%

\begin{definition}\label{defRV1}
Let $T$ be an interval exchange map. A {\it connection} is a triple 
$(u_i,v_j,m)$ consisting of a singularity of $T$, a singularity of $T^{-1}$ and a nonnegative integer which satisfy
$$ T^m (v_j) = u_i.$$
\end{definition}

Keane has proved \cite{Kea1} that an i.e.m.\ with no connection is minimal.
Having no connection means in particular that no point is both a singularity of $T$ and $T^{-1}$.

\begin{definition}\label{defRV2}
Let $T$ be an i.e.m.\ such that $u_{d-1}(T) \ne v_{d-1}(T)$. This is the case if $T$ has no connection. Define $\hat u_d = \hat v_d:= \max (u_{d-1}(T),v_{d-1}(T))$, $\wh I := (u_0(T),\hat u_d)$. The first return map of $T$ to
$\widehat I$, denoted by $\wh T$, is an i.e.m whose combinatorial data $\widehat \pi$ are canonically labeled by
the same alphabet $\A$ than $\pi$ (cf.[MMY1] p.829). The return time is $1$ or $2$.
The process   $T \mapsto \wh T$ is the {\it elementary step of the Rauzy--Veech renormalization algorithm.}
\end{definition}

%\begin{definition}\label{defRV3}
%The step $T \rightarrow \wh T$ is of {\it top type} if $u_{d-1} < v_{d-1}$, of {\it bottom type} if $u_{d-1} > v_{d-1}$. 
%
%In the top case, the singularities of $T$ and $\wh T$ are the same, and the combinatorial data $\wh \pi =: R_t (\pi)$ of $\wh T$ satisfy $\pi_t = \wh \pi_t$.
%
%In the bottom case, the singularities of $T^{-1}$ and $\wh T^{-1}$ are the same, and the combinatorial data $\wh \pi =: R_b (\pi)$ of $\wh T$ satisfy $\pi_b = \wh \pi_b$.
%\end{definition}

The combinatorial data $\wh \pi$ of $\wh T$ are irreducible.
\par
If $T$ has no connection, then $\widehat T$ has no connection.  This allows to iterate
indefinitely  the elementary step of the algorithm when $T$ has no connection.
\par

%%%%%%%%%%%%%%%%%%%%%%
\subsection{The Rauzy-Veech algorithm}\label{ssRValgo}
%%%%%%%%%%%%%%%%%%%%%%

Let $T$ be an i.e.m.\  with no connection acting on a bounded open interval $I(T)$. 
%Denote by $\cR$ the Rauzy class containing the combinatorial data $\pi$ of $T$, and by $\cD$ the associated Rauzy diagram.

\par
Starting from $T(0) :=T$, we iterate the elementary step of the Rauzy-Veech algorithm described in section \ref{ssbasicRV}. We obtain a sequence 
$(T(n))_{n\geq 0} $ of i.e.m acting on a decreasing sequence of intervals $I(n):=I(T(n))$ with the same left endpoint $u_0(T)$. For $n \geq m$, the i.e.m. $T(n)$ is the first return map of $T(m)$ to $I(n)$. We denote by $\pi (n)$ the combinatorial data of $T(n)$.

%%%%%%%%%%%%%%%%%%%%%%%%%%%%%%%%%
\subsection{The discrete time Kontsevich-Zorich cocycle}\label{ssKZ}
%%%%%%%%%%%%%%%%%%%%%%%%%%%%%%%%%
Let $T$ be an i.e.m.. 

\begin{definition}\label{defKZ0}
We denote by $\Gamma(T)$ the space of functions on $\sqcup \iat (T)$ which are constant on each $\iat (T)$. It is canonically identified with $\RA$. For a real number $r\ge 0$ we denote by 
$C^r({\bf u}(T))$ the product $ \prod_1^d C^r([u_{i-1}(T),u_i(T)])$. 
\end{definition}

\begin{definition}\label{defKZ01}
Assume that $T$ has no connection. Let $(T(n))_{n\ge 0}$ be the sequence defined by the Rauzy--Veech algorithm. 
For $n\ge m\ge 0$ and $\varphi \in  C^0({\bf u} (T(m)) )$ the  {\it special Birkhoff sum} $S(m,n)\varphi$ is the function in $C^0({\bf u} (T(n)) )$ defined by 
 $$ S(m,n)\varphi(x) = \sum_{0 \leq j <r(x)} \varphi \circ T(m )^j(x),$$
where $x\in [u_{i-1}(T(n)),u_i(T(n))]\subset I(n)$ and $r(x)$ is the return time of $x$ in $ I(n)$ 
under $T(m)$.
\end{definition}

For $m\le n\le p$, one has the cocycle relation 
$$ S(m,p)=S(n,p)\circ S(m,n) \, . $$
The operator $S(m,n)$ maps $\Gamma (T(m))$ onto $\Gamma (T(n))$. We denote by $B(m,n)$ the matrix of the 
restriction of $S(m,n)$ to $\Gamma (T(m))$ w.r.t. the canonical bases of $\Gamma (T(m))$ and $\Gamma (T(n))$.
The coefficient $B_{\alpha\beta}(m,n)$ is therefore equal to the number of visits under $T(m)$ of $I_\alpha (n)$ into $I_\beta (m)$ before returning to $I(n)$. For $x\in I_\alpha (n)$, the return time $r(x)$ is equal to $B_\alpha (m,n):=
\sum_\beta B_{\alpha\beta}(m,n)$. 
We have the partition 
\begin{equation}\label{eqpartition}
 I(m)=\bigsqcup_{\alpha\in\A}\bigsqcup_{0\le j<B_\alpha (m,n)} T(m)^j(I_\alpha^t(n))\;\hbox{mod.}\, 0\, . 
\end{equation}

The matrices $B(m,n)$ satisfy the cocycle relation 
$$ B(m,p)=B(n,p)B(m,n)\, . $$
They define the {\it (extended) Kontsevich-Zorich cocyle}. 

Let $m\ge 0$, and let $\alpha_t,\alpha_b$ be the letters of $\A$ such that $\pi_t(m)\alpha_t = \pi_b(m)\alpha_b=d$. 
For $\alpha,\beta\in\A$, let $E_{\alpha\beta}$ be the elementary matrix whose only nonzero coefficient is equal to $1$ in position $\alpha\beta$. The matrix $B(m,m+1)$ is equal to $\mathbf 1+E_{\alpha_t\alpha_b}$ if $u_{d-1}(T(m))>
v_{d-1}(T(m))$ and to $\mathbf 1+E_{\alpha_b\alpha_t}$ if $u_{d-1}(T(m))<
v_{d-1}(T(m))$. 

It follows that $B(m,n)$ belongs to $SL(\Zset^{\A})$ and has nonnegative coefficients. Moreover, for each $\alpha,\beta\in\A$, the sequence $B_{\alpha\beta}(m,n)$ is a nondecreasing function of $n$. 
It follows from \cite{MMY1} p. 832-833 that for fixed $m$ and  $n$ large enough all coefficients $B_{\alpha\beta}(m,n)$
are strictly positive. 

We define a sequence of integers $(n_k)_{k\ge 0}$ as follows: $n_0=0$ and $n_{k+1}$ is the smallest 
integer such that all coefficients of $B(n_k,n_{k+1})$ are strictly positive.

%%%%%%%%%%%%%%%%%%%%%%%%%%%%
%%%%%%%%%%%%%%%%%%%%%%%%%%%%%%
\subsection{The boundary operator}\label{Secboundary}
%%%%%%%%%%%%%%%%%%%%%%%%%%%%%%
%%%%%%%%%%%%%%%%%%%%%%%%%%%%%%

Let $\pi$ be irreducible combinatorial data over the alphabet $\A$.
\par

Define a $2d$-element set $\cS$ by
 
 $$ \mathcal S = \{ U_0 = V_0, U_1, V_1, \ldots, U_{d-1}, V_{d-1}, U_d = V_d \} .$$ 
 
 These symbols correspond to the vertices of the polygon produced by the suspension of an i.e.m $T$ with  combinatorial data $\pi$ (section \ref {ssSuspension}). 
 \par
 Going anticlockwise around the vertices (taking the gluing into account) produces a permutation $\sigma$ of $\cS$:
 
 $$
 \begin{array}{lll}
 \sigma (U_i) = V_j & {\rm with }\; \pi_b^{-1}(j+1)= \pi_t^{-1}(i+1) & {\rm for}\; 0 \leq i < d \\
 \sigma (V_{j'}) = U_{i'} & {\rm with }\; \pi_t^{-1}(i')= \pi_b^{-1}(j') & {\rm for}\; 0 < j' \leq d.
 \end{array}
 $$
  
The cycles of $\sigma$ in $\mathcal S$ are canonically associated to the marked points on the translation surface $M_T$ obtained by suspension of $T$. We will denote by $\Sigma $ the set of cycles of $\sigma$, by $s$ the cardinality of $\Sigma$ (cf. section \ref {ssSuspension}).
  
\begin{definition}
Let $\varphi \in C^0({\bf u}(T))$.  We write $\varphi(u_i(T)-0)$ (resp. $\varphi (u_i(T) +0)$) for its value at $u_i(T)$ considered as a point in $[u_{i-1}(T),u_i(T)]$ (resp.
 $[u_i(T), u_{i+1}(T)]$). We will use the convention that $\varphi (u_0(T) -0) = \varphi (u_d(T) +0) =0$.
\end{definition}
  
\begin{definition}\label{defboundary1}
Let $T$ be an i.e.m. with combinatorial data $\pi$, and let $\Sigma$ be the set of cycles
of  the associated permutation $\sigma$ of $\cS$. The {\it boundary operator} $\partial$ is the linear operator from $C^0({\bf u} (T) )$ to $\RS$ defined by
\begin{equation}\label{eqbound}
(\partial \varphi)_C := \sum _{0 \leq i \leq d, \; U_i \in C} (\varphi (u_i(T) - 0) - \varphi (u_i (T)+ 0)),
\end{equation}
for any $\varphi \in C^0({\bf u} (T) )$, $C \in \Sigma$.
\end{definition} 

\begin{remark}
The presentation of the permutation $\sigma$ is different from \cite{MMY2}. 
\end{remark}

\begin{remark}\label{trivboundary}
If $\varphi\in C^0({\bf u} (T) )$ satisfies $\varphi (u_i(T)+0)=\varphi(u_i(T)-0)$ for $0\le i\le d$ then $\partial\varphi =0$. 
\end{remark}

\begin{remark}\label{boundarysum}
Let $\varphi\in C^1({\bf u} (T) )$. One has 
$$\sum_{C\in\Sigma} (\partial\varphi)_C = \int_{u_0(T)}^{u_d(T)} D\varphi (x)dx\, . $$
\end{remark}

\begin{remark} \label{remboundary0} 
The intersection of every cycle of $\sigma$ with the ($d+1$)-subset $\{U_0,\ldots U_d \}$ of $\cS$ has at least two elements.
\end{remark}  

\begin{remark}\label{remhomology}
The name boundary operator is due to the following homological interpretation. The space 
$\Gamma(T)\subset  C^0({\bf u}(T))$ is naturally isomorphic to the first relative homology group $H_1(M_T,\Sigma,\Rset)$ of the translation surface $M_T$: the characteristic function of $\iat(T)$ corresponds to the homology class $[\zeta_{\alpha}]$ of the side $\zeta_{\alpha}$ (oriented rightwards) of the polygon giving rise to $M_T$. Through this identification, the restriction of the boundary operator to $\Gamma (T)$ is the usual boundary operator
$$\partial: H_1(M_T,\Sigma,\Rset) \longrightarrow H_0(\Sigma,\Rset) \simeq \RS .$$
\end{remark}

We recall the following statement from \cite{MMY2} (Proposition 3.2, p. 1597):

\begin{proposition}\label{propboundary}
Let $T$ be an  i.e.m.\ with combinatorial data $\pi$. 
\begin{enumerate}
\item Let $\psi\in C^0([u_0(T),u_d(T)])$. One has 
$\partial\psi = \partial (\psi\circ T)$.
\item The kernel $\Gamma_\partial (T)$ of the restriction of $\partial$ to $\Gamma (T)$ is the image of 
$\Omega (\pi)$; the image of this restriction is the  hyperplane
$\RSO:= \{x \in \RS,\; \sum_C x_C =0\}$.
\item The boundary operator $\partial\, :\, C^0({\bf u}(T)) \to \RS$ is onto.
\item Let $n\ge 0$, let $T(n)$ be the i.e.m.\ obtained from $T$ by $n$ steps of the Rauzy--Veech algorithm. 
For $\varphi\in C^0({\bf u} (T) )$ one has 
$$ \partial (S(0,n)\varphi )=\partial\varphi\, , $$
where the left--hand side boundary operator is defined using the combinatorial data $\pi (n)$ of $T(n)$.
\end{enumerate}
\end{proposition}

%%%%%%%%%%%%%%%%%%%%%%%%%%%%%%%
\subsection{Symplecticity of the Kontsevich-Zorich cocycle}\label{ssKZsymplectic}
%%%%%%%%%%%%%%%%%%%%%%%%%%%%%%%

Let $T$ be an i.e.m.\ with no connection. We use the notations of section \ref{ssKZ}.
For $0\le m\le n$, one has (see \cite{Y4}, p. 28)

\begin{equation}\label{eqsymp}
B(m,n)\Omega (\pi (m))  {}^tB(m,n)= \Omega (\pi (n))\, . 
\end{equation}
The above relation has several consequences: 
\begin{itemize}
\item The operator $^tB(m,n)^{-1}$ maps the kernel ${\rm Ker}\,\Omega (\pi (m))$ of $\Omega (\pi (m))$ onto the kernel ${\rm Ker}\,\Omega (\pi (n))$ of $\Omega (\pi (n))$. 
\item The operator $B(m,n)$ maps the image ${\rm Im}\,\Omega (\pi (m))$ of $\Omega (\pi (m))$ onto the image ${\rm Im}\,\Omega (\pi (n))$ of $\Omega (\pi (n))$.
\item The formula 
$$<\Omega (\pi (m))v, \Omega (\pi (m))w>:=  {}^tv\Omega (\pi (m))w$$
defines a symplectic structure on ${\rm Im}\,\Omega (\pi (m))$, and similarly on ${\rm Im}\,\Omega (\pi (n))$.
The operator 
$B(m,n)$ is symplectic with respect to these structures. 
\end{itemize}

Moreover (see \cite{Y4}, p. 28), one can choose for 
every irreducible combinatorial data $\pi$ a basis of ${\rm Ker}\,\Omega (\pi )$
such that the matrix of  
$$ {}^tB(m,n)^{-1}\, : \, 
{\rm Ker}\,\Omega (\pi (m))\to {\rm Ker}\,\Omega (\pi (n))$$ 
w.r.t.\ the corresponding bases is the identity matrix. 

%%%%%%%%%%%%%%%%%%%%%%%%%%%%%%%
%%%%%%%%%%%%%%%%%%%%%%%%%%%%%%%
\section{The cohomological equation}\label{homological}
%%%%%%%%%%%%%%%%%%%%%%%%%%%%%%%
%%%%%%%%%%%%%%%%%%%%%%%%%%%%%%%

%%%%%%%%%%%%%%%%%%%%%%%%%%%%%%%%%%%
\subsection{Interval exchange maps of restricted Roth Type}\label{ssRoth}
%%%%%%%%%%%%%%%%%%%%%%%%%%%%%%%%%%%%%

We recall the diophantine condition on an i.e.m.\ introduced in \cite{MMY1} and slightly modified in\cite{MMY2}.

For a matrix $B = (B_{\alpha \beta})_{\alpha,\beta \in \A}$, we define 
$ ||B || := \max_{\alpha} \sum_{\beta} | B_{\alpha \beta}|$, which is the operator norm for the $\ell_{\infty}$ norm on $\RA$.

Let $T$ be an i.e.m, with combinatorial data $\pi$.  In the space $\Gamma(T) \simeq \RA$ of piecewise constant functions on $I(T)$, we define the hyperplane
$$ \Gamma_0(T) := \{ \chi \in \Gamma(T) ,\; \int \chi(x) \; dx = 0\}.$$

Assuming that  $T$ has no connection, we also define the stable subspace

$$ \Gamma_s(T) := \{ \chi \in \Gamma(T), \exists \,\sigma >0,\; ||B(0,n)\chi|| = \mathcal {O} (||B(0,n)||^{-\sigma}) \}.$$

As $\Gamma_s(T)$ is finite-dimensional, there exists an exponent $\sigma >0$ which works for every $\chi\in\Gamma_s(T)$. We fix such 
an exponent in the rest of the paper. 

The subspace $ \Gamma_s(T)$ is contained in ${\rm Im}\,\Omega (\pi )$, and is an isotropic subspace of this symplectic space.
We have $\Gamma_s(T(n)) = B(0,n)\Gamma_s(T)$ for $n \geq 0$.

\medskip

We introduce four conditions. The sequence $(n_k)$ in the first condition has been defined in section  \ref{ssKZ}.

\bigskip

\hspace{3mm} (a) \hspace{3mm} For all $\tau>0$, $||B(n_k, n_{k+1})|| = \mathcal {O} (||B(0,n_k)||^{\tau})$. \\

\hspace{3mm} (b) \hspace{3mm} There exists $\theta >0$ such that
$$||B(0,n)_{|\Gamma_0(T)}|| = \mathcal {O} (||B(0,n)||^{1-\theta})\,.$$

\hspace{3mm} (c) \hspace{3mm}  For $0 \leq m\leq n$, let 
\begin{eqnarray*}
B_s(m,n)&:& \Gamma_s(T(m)) \to \Gamma_s(T(n)) ,\\
  B_{\flat}(m,n)&: &\Gamma(T(m)) /\Gamma_s(T(m)) \to \Gamma(T(n)) /\Gamma_s(T(n))  
  \end{eqnarray*}

be the operators  induced by $B(m,n)$.
We ask that, for all $\tau >0$,
$$||B_s(m,n)||= \mathcal {O} (||B(0,n)||^{\tau}), \quad ||(B_{\flat}(m,n))^{-1}|| = \mathcal {O}
(||B(0, n)||^{\tau})\,.$$

\hspace{3mm} (d) \hspace{3mm} $\dim \Gamma_s (T) =g$.

\begin{definition}\label{defRoth2}
An i.e.m. $T$ with no connection is of {\it
restricted Roth type} if the four
conditions (a), (b), (c), (d) are satisfied. 
\end{definition}

\begin{remark}\label{remRoth0} 
In \cite{MMY1}, condition (a) had a slightly different (but equivalent) formulation. Interval exchange maps satisfying the first three conditions (a),(b),(c) were said to be of {\it Roth  type}.
\end{remark}

\begin{remark}\label{remRoth1}
 An i.e.m.\ $T$ with no connection satisfying condition (b)
  is uniquely ergodic.
\end{remark}
\begin{remark}\label{remRoth2}
For any combinatorial data, the set of i.e.m of restricted Roth type has full measure.
First, it is obvious that almost all i.e.m. have no connection. That condition (c) is almost surely satisfied is a consequence of Oseledets theorem applied to the Kontsevich-Zorich cocycle. Condition (d) follows from the hyperbolicity of the (restricted) Kontsevich-Zorich cocycle, proved by Forni (\cite {For2}). 
A proof that condition (a) has full measure is provided in \cite{MMY1}, but much better diophantine
estimates were later obtained in \cite{AGY}. Finally, the fact that condition (b) has full measure is a  consequence from the fact that the
larger Lyapunov exponent of the Kontsevich-Zorich cocycle is simple (see \cite{Ve4}).
\end{remark}

%%%%%%%%%%%%%%%%%%%%%%%%%%%%%%%%%%%%%%%%%%%%%%%%%%%%%%%%%%%%
\subsection{Some consequences of the hyperbolicity of the Kontsevich-Zorich cocycle}\label{ssremarque}
%%%%%%%%%%%%%%%%%%%%%%%%%%%%%%%%%%%%%%%%%%%%%%%%%%%%%%%%%%%

Assume that $T$ satisfies condition (d) of the previous section. Then the subspace 
$\Gamma_s(T)$ is a lagrangian subspace of the symplectic space ${\rm Im}\,\Omega (\pi )$.
 Choose a basis $\mathfrak B_0$ of $\Gamma(T)$ such that  the first $g$ vectors of 
 $\mathfrak B_0$ form a basis of  $\Gamma_s(T)$ and the first $2g$ vectors of $\mathfrak B_0$
  form a symplectic basis of ${\rm Im}\,\Omega (\pi )$.
 
For each $n \geq 0$, choose a basis $\mathfrak B_n$ of $\Gamma (T(n))$ such that 

\begin{itemize}
\item the first $g$ vectors of $\mathfrak B_n$ form a basis  of $\Gamma_s (T(n))$;
\item the first $2g$ vectors of $\mathfrak B_n$ form a symplectic basis of 
${\rm Im}\,\Omega (\pi(n) )$;
\item consider the operator from $\Gamma(T) / {\rm Im}\,\Omega (\pi )$ to 
$\Gamma(T(n)) / {\rm Im}\,\Omega (\pi (n))$ induced by $B(0,n)$. Its 
matrix representation w.r.t.\ the bases given by the last  $(s-1)$
vectors of $\mathfrak B_0$ , $\mathfrak B_n$ is the identity matrix (see section \ref{ssKZsymplectic});
\item The sequence of bases $\mathfrak B_n$  has bounded distortion: there exists a constant $c>0$ independent of $n$ such that, 
for all $n\ge 0$, $a\in \Rset^{\mathfrak B_n}$
$$ c^{-1}\Vert a\Vert_\infty \le \Vert \sum_{\chi\in \mathfrak B_n}a_\chi\chi\Vert_\infty\le c\Vert a\Vert_\infty\; . $$
Achieving this is possible since the orthogonal group acts transitively on the set of lagrangian subspaces of $\Rset^{2n}$. 
%the $\ell_{\infty}$-norm on $\Gamma(T(n))$ defined by $\mathfrak B_n$ is equivalent to any reference norm, \emph{uniformly in $n$}.
%(there are only finitely many possible combinatorial data $\pi(n)$, but infinitely many integers $n$!)
\end{itemize}

For $m \le n$, the operator $B(m,n)$ sends ${\rm Im}\,\Omega (\pi(m) )$ onto ${\rm Im}\,\Omega (\pi(n) )$ and is symplectic for the natural symplectic forms on these two spaces. The matrix of this operator with respect to the bases $\mathfrak B_m$, $\mathfrak B_n$ has the block triangular form

$$ \left ( \begin{array} {ccc} b_s(m,n) & \star &\star\\ 0 & b_u(m,n)& \star\\ 0&0&\mathbf1 \end{array} \right ) $$

The block $ \left ( \begin{array} {cc}  b_u(m,n)& \star \\ 0&\mathbf1 \end{array} \right ) $ is the matrix of the operator  $B_{\flat}(m,n)$ ,  with respect to the bases induced by 
$\mathfrak B_m$, $\mathfrak B_n$. Let us write $ \left ( \begin{array} {cc}  b_u^{-1}(m,n)& \mathfrak b(m,n) \\ 0&\mathbf1 \end{array} \right ) $ for the inverse of this block.

From the symplecticity of the restriction of $B(m,n)$ to ${\rm Im}\,\Omega (\pi(m) )$ , one has

\begin{equation}\label{eqbu1}
 b_u^{-1}(m,n) = ^t b_s(m,n).
 \end{equation}

When $m=0$, this implies that we have, from the definition of $\Gamma_s(T)$

\begin{equation}\label{eqbu2}
 \Vert b_u^{-1}(0,n) \Vert \leq C \Vert B(0,n) \Vert^{-\sigma} .
 \end{equation}

The cocycle relation $ B(m,p) = B(n,p) B(m,n)$ ($m \leq n \leq p$) gives

\begin{equation}\label{eqbu3}
\mathfrak b(m,p) =  \mathfrak b(m,n) + ^t b_s(m,n) \mathfrak b(n,p) .
 \end{equation}

Taking $m=0$  this leads to

\begin{equation}\label{eqbu4}
\mathfrak b(0,p) =  \sum_{n=0}^{p-1}  \, ^t b_s(0,n) \mathfrak b(n,n+1) .
 \end{equation}

which gives $\Vert \mathfrak b(0,p)  \Vert \leq C$.

\smallskip

More generally, we also get 

\begin{equation}\label{eqbu5}
\mathfrak b(m,p) =  \sum_{n=m}^{p-1}  \, ^t b_s(m,n) \mathfrak b(n,n+1) .
 \end{equation}

Assuming the first half of condition (c) in the last section

$$ \Vert b_s(m,n) \Vert = \mathcal {O}(  \Vert B(0,n) \Vert^\tau ),\qquad \forall \tau >0, $$

we obtain

\begin{equation}\label{eqbu6}
\Vert \mathfrak b(m,p) \Vert = \mathcal {O}(  \sum_{n=m}^{p-1}  \Vert B(0,n) \Vert^\tau ) = 
\mathcal {O}(\Vert B(0,p) \Vert^\tau), \quad \forall \tau >0.
 \end{equation}

As we have also $ \Vert b_u^{-1}(m,p) \Vert = \Vert  ^t b_s(m,p) \Vert = \mathcal {O}(  \Vert B(0,p) \Vert^\tau )$, we conclude that 

\begin{proposition}\label{prophalfc}
When $T$ satisfies condition (d) of section \ref{ssRoth}, the first part of condition (c) implies the second part.
\end{proposition}

%%%%%%%%%%%%%%%%%%%%%%%%%%%%
\subsection{Previous results on the cohomological equation}\label{ssprevious cohomological}
%%%%%%%%%%%%%%%%%%%%%%%%%%%%%

To state results on the cohomological equation, we consider H\"older regularity conditions, like in Appendix A of \cite{MMY2}, rather than the bounded variation setting of [MMY1] and of the main text of \cite{MMY2}.

\smallskip

\begin{definition}\label{defcohom1}
Let $T$ be an i.e.m.\ and let $r\geq 0$ be a real number. We denote by $C^r_{0} ({\bf u} (T))$ the kernel of the boundary operator $\partial$ in $\CruT$.
%, by $\Gamma_\partial (T)$ the intersection $\Gamma_\partial (T):= \Gamma (T)\cap 
%C^r_{0} ({\bf u} (T))$
\end{definition}

\begin{theorem}\label{thmcohom1}
Let $T$ be an i.e.m. of restricted Roth type. Let 
$\Gamma_u(T)$ be a subspace of $\Gamma_\partial (T)$
which is supplementing $\Gamma_s(T)$ in $\Gamma_\partial (T)$. Let $r $ be a real
 number $>1$. There exist bounded linear operators $L_0: \varphi \mapsto \psi$ from 
 $C^r_{0} ({\bf u} (T))$ to $C^0([u_0(T),u_d(T)])$ and $L_1: \varphi \mapsto \chi$ from 
$C^r_{0} ({\bf u} (T))$ to $\Gamma_u(T)$ such that any $\varphi \in C^r_{0} ({\bf u} (T))$ satisfies
$$ \varphi = \chi + \psi \circ T - \psi\; , \;\; \int_{u_0(T)}^{u_d(T)}\psi (x)dx = 0\; . $$
\end{theorem}

The proof is a combination of the proof of Theorem A in \cite{MMY1}, of 
Remark 3.11 of \cite{MMY2}, and of Appendix A of the same paper. The main steps of the proof will be recalled in the next sections, where a stronger result, stated below, is proved.

\begin{remark}
By Proposition \ref{propboundary} part (a), the coboundary of a continuous function belongs to the kernel of the boundary 
operator.
\end{remark}

\begin{remark}
The operators $L_0$ and $L_1$ are uniquely defined by the conclusions of the theorem. For $L_1$, it follows from Remark \ref{remarkchi} below. Then, 
as $T$ is minimal, $L_0$ is also uniquely defined. 
\end{remark}

%%%%%%%%%%%%%%%%%%%%%%%%%%%%%%%%%%%
\subsection{H\"older regularity of the solutions of the cohomological equation}\label{ssHolder}
%%%%%%%%%%%%%%%%%%%%%%%%%%%%%%%%%%

Our main result is an improvement on the regularity of the solutions of the cohomological equation.

\smallskip

\begin{theorem} \label{thmHolder1}
In the setting of Theorem \ref{thmcohom1}, there exists $\bar \delta >0$ such that 
the operator $L_0$ takes its value in the space $C^{\bar \delta}([u_0(T),u_d(T)])$ of 
H\"older continuous functions of exponent $\bar \delta$. The operator $L_0$ is bounded from $C_0^r({\bf u}(T))$ to $C^{\bar \delta}([u_0(T),u_d(T)])$. The exponent $\bar \delta$ depends only  on $r$ and the constants $\theta$, 
$\sigma$ appearing in section \ref{ssRoth}.
\end{theorem}

The rest of the paper is devoted to the proof of this theorem.

%%%%%%%%%%%%%%%%%%%%%%%%%%%%%%%%%%%
\subsection{Estimates for special Birkhoff sums with $C^1$ data}\label{ssC1}
%%%%%%%%%%%%%%%%%%%%%%%%%%%%%%%%%%

The following result is of independent interest. 

\begin{theorem}\label{thmcohom2}
Let $T$ be an i.e.m. which satisfies conditions (a), (c), (d) of section \ref{ssRoth}. Let  
$\Gamma_u(T)$ be a subspace of $\Gamma_\partial (T)$
which is supplementing $\Gamma_s(T)$ in $\Gamma_\partial (T)$. The operator $L_1: \varphi \mapsto \chi$ of Theorem \ref{thmcohom1} extends to a bounded operator
 from $C^1_{0} ({\bf u} (T))$ to $\Gamma_u(T)$. For $\varphi \in C^1_{0} ({\bf u} (T))$, the function $\chi= L_1(\varphi) \in \Gamma_u(T)$ is characterized by the property that the special Birkhoff sums of $\varphi -\chi$ satisfy, for any $\tau >0$

 $$ \Vert S(0,n) (\varphi -\chi) \Vert_{C^0} \leq C(\tau) \Vert \varphi \Vert_{C^1} \Vert B(0,n)\Vert^\tau .$$ 

\end{theorem}

\begin{remark}\label{remarkchi} The function $\chi\in\Gamma_u(T)$ is characterized by the inequality of the theorem. 
We have to show that if a function $\chi\in\Gamma_u(T)$  satisfies 
$$ \Vert B(0,n)\chi\Vert\le C_\tau\Vert B(0,n)\Vert^\tau\Vert\chi\Vert\; , $$
for all $\tau >0$, then $\chi =0$. 

Indeed, choose bases $\mathfrak B_n$ of $\Gamma (T(n))$ as in section \ref{ssremarque}, and write $\chi = 
\left ( \begin{array} {c}  \star \\  \chi_u \\ 0 \end{array} \right ) $ 
in $\mathfrak B_0$. Then the image $B(0,n)\chi$ has the 
form $\left ( \begin{array} {c}  \star \\  b_u(0,n)\chi_u \\ 0 \end{array} \right ) $ in $\mathfrak B_n$. 
Therefore, we have, for all $\tau >0$ 
$$ \Vert b_u(0,n)\chi_u \Vert \le C_\tau \Vert B(0,n)\Vert^\tau \Vert\chi\Vert\; . $$
From inequality (\ref{eqbu2}), we must have $\chi_u=0\;$ thus, as $\chi\in\Gamma_u(T)$, $\chi=0$. 
\end{remark}

\begin{remark}\label{remtime} The inequality of the theorem for special Birkhoff sums implies a similar inequality for general Birkhoff sums: for all 
$\tau >0$, 

$$ \Vert \sum_{j=0}^{N-1} (\varphi -\chi) \circ T^j \Vert_{C^0} \leq C(\tau) \Vert \varphi \Vert_{C^1} N^\tau .$$

This follows from the time decomposition in section \ref{sstime}. 
\end{remark}

\begin{proof} The proof of the Theorem \ref{thmcohom2} is a variant of the proof of the Theorem at p.\ 845 of \cite{MMY1}.

Recall the following result from \cite{MMY1}, p. 835.
\begin{lemma}\label{lemHolder1}
Let $T$ be an i.e.m.\ with no connection and let $n\ge m \ge 0$. 
\begin{enumerate}
\item One has 
$$\min_{\alpha} |\iat(T(n))| \leq |I(T(m))| || B(m,n)||^{-1} \leq d
\max_{\alpha}  |\iat(T(n))|\; .$$

\item Assume that $T$ satisfies condition (a)  of section \ref{ssRoth}. Then, for any $\tau >0$, there exists $C_{\tau}$ such that
$$\max_{\alpha}  |\iat(T(n))| \leq C_{\tau} \,|| B(0,n) ||^{\tau} \min_{\alpha} |\iat(T(n))| 
 .$$
\end{enumerate}
\end{lemma}

Let $\varphi \in C^1_{0} ({\bf u} (T))$. Let $n \geq 0$. From the trivial estimate
$$ \Vert S(0,n) D\varphi \Vert_{C^0} \leq \Vert D\varphi \Vert_{C^0} \Vert B(0,n) \Vert $$
and Lemma \ref{lemHolder1}, we get, for $\alpha \in \A$, $x,y \in \iat(T(n))$

 \begin{equation}\label{eq4a}
  \vert S(0,n) \varphi (y) -   S(0,n) \varphi (x) \vert \leq C_{\tau}  \Vert D\varphi \Vert_{C^0}  \Vert B(0,n) \Vert^\tau.
  \end{equation}

There exists a function $\Phi (n)\in \Gamma_\partial (T(n))$ such that $S(0,n) \varphi - \Phi (n)$ is continuous on $I(T(n))$ and vanishes at both 
endpoints of this interval. 
Indeed, let $\Phi (n)$ be the function in $\Gamma (T(n))$ such that $S(0,n) \varphi - \Phi (n)$ is continuous on $I(T(n))$ and vanishes at $u_0(T(n))$. 
As $\varphi\in C^1_{0} ({\bf u} (T))$, $D\varphi$ has mean zero, hence $S(0,n) (D\varphi) = D(S(0,n) \varphi)$ has also mean zero; thus 
$S(0,n) \varphi - \Phi (n)$ vanishes at $u_d(T(n))$. Then $S(0,n) \varphi - \Phi (n)$
belongs to $C^1_{0} ({\bf u} (T(n)))$. As $S(0,n) \varphi$ also belongs to $C^1_{0} ({\bf u} (T(n)))$, we conclude that 
$\Phi (n)\in \Gamma_\partial (T(n))$.

For $\alpha \in \A$, $x \in \iat(T(n))$, we get from (\ref{eq4a})

 \begin{equation}\label{eq41}
 \vert S(0,n) \varphi (x) - \Phi_\alpha (n) \vert \leq C_{\tau}  \Vert D\varphi \Vert_{C^0}  \Vert B(0,n) \Vert^\tau.
\end{equation}

We will show that there exists $\chi \in \Gamma_u(T)$ such that 

\begin{equation}\label{eq42}
 \Vert \Phi(n_\ell) - B(0,n_\ell) \chi \Vert \leq C'_{\tau}  \Vert \varphi \Vert_{C^1}  \Vert B(0,n_\ell) \Vert^\tau,
\end{equation}

and that it satisfies $\Vert \chi \Vert \leq C  \Vert \varphi \Vert_{C^1}$.

\smallskip

We write $C_\tau$ for various constants depending only on $\tau$ (through conditions (a), (c) of section \ref{ssRoth}). From $S(0,n_{\ell +1}) = S(n_\ell,n_{\ell +1})  \circ S(0,n_\ell)$ and (\ref{eq41}), we get, for $\ell \geq 0$

  \begin{equation}\label{eq43}
\Vert \Phi(n_{\ell +1}) - B(n_\ell, n_{\ell +1}) \Phi(n_\ell) \Vert  \leq C_{\tau}  \Vert D\varphi \Vert_{C^0}  \Vert B(0,n_\ell) \Vert^\tau.
\end{equation}

Define $\Lambda_{\ell +1} $ to be the class mod. $\Gamma_s(T(n_{\ell +1}))$ of 
$\Phi(n_{\ell +1}) - B(n_\ell, n_{\ell +1}) \Phi(n_\ell)$.  We define $\chi$ to be the unique element in $\Gamma_u(T)$ which is equal mod. $\Gamma_s(T)$ to 

$$ \Phi(0) + \sum_{\ell >0} B_\flat^{-1} (0,n_\ell) \Lambda_\ell .$$

We will show shortly that the series is convergent. Notice that, if $\chi$ is defined in this way, the image $B(0,n_k) \chi$ is equal mod. $\Gamma_s(T(n_k))$ to

$$ \Phi(n_k) + \sum_{\ell >k} B_\flat^{-1} (n_k,n_\ell) \Lambda_\ell .$$

To estimate the norm of the series $\sum_{\ell >k} B_\flat^{-1} (n_k,n_\ell) \Lambda_\ell$, we split the sum into two parts

\begin{itemize}
\item As long as $\Vert B(0, n_\ell) \Vert ^{\sigma /2} \leq \Vert B(0,n_k) \Vert$, we will use (from condition (c) in section \ref{ssRoth} and (\ref{eq43}))
 \begin{equation}\label{Bflat}
  \Vert B_\flat^{-1} (n_k,n_\ell) \Lambda_\ell \Vert  \leq C_\tau \Vert D\varphi \Vert _{C^0}  \Vert B(0,n_\ell)\Vert ^\tau.
  \end{equation}
  The total contribution of this part of the series is at most
  $$ C_\tau \Vert D\varphi \Vert _{C^0}  \Vert B(0,n_k)\Vert ^\tau.$$
\item When $\Vert B(0, n_\ell) \Vert ^{\sigma /2} > \Vert B(0,n_k) \Vert$, we proceed differently. We use notations of section \ref{ssremarque}. As $\Phi(n)$ belongs to 
$\Gamma_\partial (T(n))$, the function $\Phi(n_{\ell +1}) - B(n_\ell, n_{\ell +1}) \Phi(n_\ell)$ belongs to $\Gamma_\partial (T(n_{\ell +1}))$. Therefore we have, mod. $\Gamma_s(T(n_k))$

\begin{align*}
 \Vert B_\flat^{-1} (n_k,n_\ell) \Lambda_\ell \Vert &=\Vert  b_u^{-1}(n_k,n_\ell) \Lambda_\ell \Vert\\   
 & \leq \Vert b_u(0,n_k) \Vert \Vert   b_u^{-1}(0,n_\ell)   \Vert \Vert   \Lambda_\ell \Vert\\
 &    \leq \Vert B(0,n_k) \Vert \Vert   ^t b_s (0,n_\ell)   \Vert \Vert   \Lambda_\ell \Vert\\
  & \leq    \Vert B(0,n_k) \Vert    \Vert B(0,n_\ell) \Vert^{-\sigma}   \Vert   \Lambda_\ell \Vert \\
  & \leq  C  \Vert B(0,n_\ell) \Vert^{-\sigma/3}     \Vert D \varphi \Vert_{C^0}.                  
   \end{align*}                                        
\end{itemize}

This shows that the series $\sum_{\ell >k} B_\flat^{-1} (n_k,n_\ell) \Lambda_\ell$
is convergent and that we have, mod. $\Gamma_s(T(n_k))$

 \begin{equation}\label{eq44}
\Vert B(0,n_k) \chi - \Phi(n_k) \Vert \leq C_\tau \Vert D\varphi \Vert _{C^0}  \Vert B(0,n_k)\Vert ^\tau.
\end{equation}

Taking $k=0$ and observing that $\Vert\Phi (0)\Vert\le C\Vert\varphi\Vert_{C^0}$, one obtains 
$\Vert\chi\Vert\le C\Vert\varphi\Vert_{C^1}\; .$ 

To show that (\ref{eq42}) holds in $\Gamma_\partial (T(n_k))$ (and not only mod. $\Gamma_s(T(n_k))$) we may assume that $\chi =0$ and proceed as
follows. Write $S(0,n_k)\varphi = \varphi_k - \chi_k$ with $\chi_k\in \Gamma_s(T(n_k))$ and 

\begin{equation}\label{eq45}
\Vert \varphi_k\Vert_{C^0} \le C_\tau \Vert B(0,n_k)\Vert ^\tau \Vert D\varphi\Vert_{C^0}\; . 
\end{equation}

Define 

$$ \Delta_{0}:= \chi_{0}\, , \;\;
\Delta_{k}:= \chi_{k}-B(n_{k-1},n_{k})\chi_{k-1} = \varphi_k - S(n_{k-1},n_{k})\varphi_{k-1}\in \Gamma_s(T(n_k))\, .  $$

Using condition (a) of section \ref{ssRoth}, one obtains 

\begin{equation}\label{eq46}
\Vert \Delta_0  \Vert\le C \Vert\varphi\Vert_{C^0}\, , \;\;
\Vert \Delta_k  \Vert\le C_\tau \Vert B(0,n_k)\Vert^\tau\Vert D\varphi\Vert_{C^0}\;\hbox{for }k>0\, . 
\end{equation}

In the formula 

$$ S(0,n_k)\varphi = \varphi_k-\sum_{\ell\le k} B(n_\ell,n_k)\Delta_\ell $$

one has  from condition (c) in section \ref{ssRoth}

$$  \Vert B(n_\ell,n_k) \Delta_\ell \Vert \leq C_\tau \Vert B(0,n_k) \Vert^{\tau}  \Vert  \varphi\Vert_{C^1}. $$

We conclude that (\ref{eq42}) holds. 
The proof of the theorem is complete. 
\end{proof}

%%%%%%%%%%%%%%%%%%%%%%%%%%%%%%%%%%%
\subsection{Estimates for special Birkhoff sums with $C^r$ data}\label{ssBirkhoff}
%%%%%%%%%%%%%%%%%%%%%%%%%%%%%%%%%%

The proof of the following lemma can be found in Appendix A of \cite{MMY2}. 
% (it is the only part of the proof where the H\"older setting requires a modification w.r.t. the bounded variation setting).

\begin{lemma}\label{lemHolder2}
Let $T$ be an i.e.m. with no connection. Assume that  $T$ satisfies conditions (a) and (b) of section \ref{ssRoth}. 
Let $\rho$ be a positive real number. There exists $\delta >0$, depending on $\rho$ and on the
 constant $\theta$ in condition (b), such that, for any function 
 $\varphi \in C^{\rho}({\bf u}(T))$ of mean $0$ and any $k \geq 0$, the special Birkhoff sum
  $S (0,n_k) \varphi$   satisfies, 
$$ || S(0,n_k) \varphi ||_{C^0} \leq C || B(0,n_k) ||^{1-\delta} || \varphi ||_{C^{\rho}}.$$
The constant $C$ depends on $\rho$, $\theta$ and the other implied constants in 
conditions (a) and (b) of \ref{ssRoth}.
\end{lemma} 

The following lemma improves the conclusion of Theorem \ref{thmcohom2} for functions in $C_0^{r}({\bf u}(T))$.

\begin{lemma}\label{lemHolder3}
Let $r=1+\rho \in (1,2)$, $T$ and $\Gamma_u(T)$ be as in the statement of Theorem \ref{thmcohom1}. 
There exists $\delta_1 >0$, depending only on $r$ and the constants $\theta,\,\sigma$ of section \ref{ssRoth} 
such that for any $\varphi\in C_0^{r}({\bf u}(T))$, the function 
$\chi=L_1(\varphi )$ of Theorem \ref{thmcohom2} satisfies 
  $$ || S(0,n_k)  (\varphi - \chi) ||_{C^0} \leq C || B(0,n_k) ||^{-\delta_1} || \varphi ||_{C^{r}}\, .$$
The constant $C$ depends on $\rho$, $\theta$, $\sigma$ and the other implied constants in 
section \ref{ssRoth}.
\end{lemma}

\begin{proof}
We adapt the proof of Theorem \ref{thmcohom2}.  
Observe that, if $\varphi \in C_0^{r}({\bf u}(T))$, the function 
$D\varphi \in C^{\rho}({\bf u}(T))$ has mean $0$ (cf. Remark \ref{boundarysum}).
\par

From Lemma \ref{lemHolder2} and Lemma \ref{lemHolder1}  we get, for $\alpha \in \A$, $x,y \in \iat(T(n))$

 \begin{equation*}
  \vert S(0,n) \varphi (y) -   S(0,n) \varphi (x) \vert \leq C  \Vert D\varphi \Vert_{C^\rho}  \Vert B(0,n) \Vert^{-\delta/2}.
  \end{equation*}

Define the sequence $\Phi (n)$ as in the proof of Theorem \ref{thmcohom2}. 
Instead of (\ref{eq41}) we get 

\begin{equation}
\vert S(0,n) \varphi (x) - \Phi_\alpha (n) \vert \leq C  \Vert D\varphi \Vert_{C^\rho}  \Vert B(0,n) \Vert^{-\delta/2}.
\end{equation}

Then the sequence $\Phi (n_\ell)$ satisfies (instead of (\ref{eq43}))

\begin{equation}
\Vert \Phi(n_{\ell +1}) - B(n_\ell, n_{\ell +1}) \Phi(n_\ell) \Vert  \leq C  \Vert D\varphi \Vert_{C^\rho}  \Vert B(0,n_\ell) \Vert^{-\delta/3}.
\end{equation}

Define then $\Lambda_\ell$ as in the proof of Theorem \ref{thmcohom2}. We have 

\begin{equation}\label{eqLambda}
\Vert \Lambda_\ell \Vert  \leq C  \Vert D\varphi \Vert_{C^\rho}  \Vert B(0,n_\ell) \Vert^{-\delta/3}.
\end{equation}

The series $\sum_{\ell >k} B_\flat^{-1} (n_k,n_\ell) \Lambda_\ell$ is again split into two parts: 

\begin{itemize}
\item as long as $\Vert B(0, n_\ell) \Vert ^{\sigma /2} \leq \Vert B(0,n_k) \Vert$, we will use (from condition (c) in section \ref{ssRoth} and (\ref{eqLambda}))
 \begin{equation}
  \Vert B_\flat^{-1} (n_k,n_\ell) \Lambda_\ell \Vert  \leq C \Vert D\varphi \Vert _{C^\rho}  \Vert B(0,n_\ell)\Vert ^{-\delta/4}.
  \end{equation}
  The total contribution of this part of the series is at most
  $$ C \Vert D\varphi \Vert _{C^\rho}  \Vert B(0,n_k)\Vert ^{-\delta/5}.$$
\item When $\Vert B(0, n_\ell) \Vert ^{\sigma /2} > \Vert B(0,n_k) \Vert$, we obtain as in the proof of Theorem \ref{thmcohom2}

\begin{equation}
 \Vert B_\flat^{-1} (n_k,n_\ell) \Lambda_\ell \Vert  \leq  C  \Vert B(0,n_\ell) \Vert^{-\sigma/3}     \Vert D \varphi \Vert_{C^0}.                  
   \end{equation}                                        
\end{itemize}

We may assume that $\delta <\sigma$. We obtain instead of (\ref{eq44})

\begin{equation}
\Vert B(0,n_k) \chi - \Phi(n_k) \Vert \leq C \Vert D\varphi \Vert _{C^\rho}  \Vert B(0,n_k)\Vert ^{-\delta/5}.
\end{equation}

This gives the required estimate mod. $\Gamma_s (T(n_k))$. To get the full estimate we assume that $\chi =0$ and define $\varphi_k$ and 
$\Delta_k$ as before. Instead of (\ref{eq45}) we have 

\begin{equation}
\Vert \varphi_k\Vert_{C^0} \le C \Vert D\varphi\Vert_{C^\rho} \Vert B(0,n_k)\Vert ^{-\delta/5} \; . 
\end{equation}

The vector $\Delta_k$ satisfies (instead of (\ref{eq46})) 

\begin{equation}\label{eqDelta}
\Vert \Delta_0\Vert \le C\Vert\varphi\Vert_{C^0}\, , \;\;
\Vert \Delta_k\Vert \le C \Vert D\varphi\Vert_{C^\rho} \Vert B(0,n_k)\Vert ^{-\delta/6} \,\hbox{for }k>0\, . 
\end{equation}

In the formula 

$$ S(0,n_k)\varphi = \varphi_k-\sum_{\ell\le k} B(n_\ell,n_k)\Delta_\ell $$

we estimate again the terms in two different ways. 

\begin{itemize}
\item if $\Vert B(0,n_\ell )\Vert \leq \Vert B(0,n_k)\Vert^{\sigma/2}$, we have, as $\Delta_\ell\in \Gamma_s(T(n_\ell))$

\begin{eqnarray*}
 \Vert B(n_\ell,n_k)\Delta_\ell\Vert &\leq &C \Vert B(0,n_k)\Vert^{-\sigma}\Vert B(0,n_\ell)^{-1}\Delta_\ell \Vert \\
& \leq  & C\Vert B(0,n_k)\Vert^{-\sigma} \Vert B(0,n_\ell) \Vert \Vert \varphi\Vert_{C^r} \\
& \leq  & C\Vert B(0,n_k)\Vert^{-\sigma/2} \Vert \varphi\Vert_{C^r} .
\end{eqnarray*}

\item if $\Vert B(0,n_\ell )\Vert > \Vert B(0,n_k)\Vert^{\sigma/2}$, we get from (\ref{eqDelta}) and condition (c) 
in section \ref{ssRoth} 

$$  \Vert B(n_\ell,n_k)\Delta_\ell\Vert \leq C \Vert D\varphi\Vert_{C^\rho} \Vert B(0,n_k)\Vert^{-\delta/7}\; . $$

\end{itemize}

With $\delta<\sigma$ we obtain the conclusion of the Lemma with $\delta_1=\delta/8$. 
 \end{proof}

%%%%%%%%%%%%%%%%%%%%%%%%%%%%%%%%%%%
\subsection{Time decomposition}\label{sstime}
%%%%%%%%%%%%%%%%%%%%%%%%%%%%%%%%%%
 
We recall from \cite{MMY1} p. 840 how to decompose orbits in order to estimate general Birkhoff sums from 
special Birkhoff sums. 

Let $T$ be an i.e.m.\ with no connection and let $x\in I(T)$, $N>0$. We will decompose the finite orbit $(T^j(x))_{0\le j< N}$.

Let $y$ be the point of this orbit which is closest to $0$. We divide the orbit into a positive part $(T^j(y))_{0\le j< N^+}$ and a negative part 
$(T^j(y))_{N^-\le j<0}$ (with $N^+-N^-=N)$. 

Let $k\ge 0$ be the largest integer such that at least one of the points $T^j(y)$, $0<j<N^+$, belongs to $I(T(n_k))$. Because 
$T(n_k)$ is the first return map of $T$ into $I(T(n_k))$, the points $T^j(y)$, $0<j<N^+$ which belong to 
$I(T(n_k))$ are precisely $T(n_k)(y).\ldots , (T(n_k))^{b(k)}(y)$ for some $b(k)>0$.

Define by decreasing induction $y(\ell)$, $b(\ell)$ for $0\le \ell <k$ as follows. 
Define  $y(k-1)=(T(n_k))^{b(k)}(y)$. The induction hypothesis is that $y(\ell)=T^{N(\ell)}(y)$ is the last point of the orbit 
$(T^j(y))_{0\le j<N^+}$ which belongs to $I(T(n_{\ell+1}))$. The points of the orbit $(T^j(y))_{N(\ell )\le j<N^+}$
which belong to $I(T(n_\ell))$ are $y(\ell), \ldots , (T(n_\ell))^{b(\ell)}(y(\ell))=:y(\ell -1)$ for some $b(\ell )\ge 0$. 

At the end, the point  $y(0)= T^{N(0)}(y)$ satisfies $N(0)<N^+$ and $T^j(y)\notin I(T(n_1))$ for $N(0)<j<N^+$. 
We set $b(0)=N^+-N(0)$ and $y(k)=y$. 

For a function $\varphi$ on $I(T)$ the Birkhoff sum of order $N^+$ of $\varphi$ at $y$ decomposes as 
\begin{equation}\label{eqtimedec}
\sum_{0\le j<N^+} \varphi (T^j(y)) = \sum_{\ell =0}^k \sum_{0\le b <b(\ell)} S(0,n_\ell )\varphi ((T(n_\ell))^b(y(\ell)))\; .
\end{equation}

As $b(\ell )$, for $0\le \ell \le k$, is at most equal to the maximal return time into $I(T(n_{\ell +1}))$ under $T(n_\ell)$ 
one has 
\begin{equation}\label{eqbl}
b(\ell) \le \Vert B(n_\ell, n_{\ell +1})\Vert \; . 
\end{equation}

The negative part of the orbit $(T^j(y))_{N^-\le j<0}$ is decomposed in a similar way. 

\begin{remark} Formulas (\ref{eqtimedec}) and (\ref{eqbl}) imply the estimate in Remark \ref{remtime}.
\end{remark}

%%%%%%%%%%%%%%%%%%%%%%%%%%%%%%%%%%%
\subsection{Space decomposition}\label{ssspace}
%%%%%%%%%%%%%%%%%%%%%%%%%%%%%%%%%%

Recall from section \ref{ssKZ}, equation (\ref{eqpartition}),  the partition
(for $k\ge 0$):
$$
 \left({\mathcal P}(k)\right)\qquad I(T)=\bigsqcup_{\alpha\in\A}\bigsqcup_{0\le j<B_\alpha (0,n_k)} T^j(I_\alpha^t(T(n_k)))\;\hbox{mod.}\, 0\, . 
$$
Let $x_-<x_+$ be two distinct points in $I(T)$. Let $k$ be the smallest integer such that $(x_-,x_+)$  contains at least  one  interval of the partition ${\mathcal P}(k)$. Let $J^{(k)}(1), \ldots, J^{(k)}(b(k))$ (with $b(k) > 0$) be the intervals
of ${\mathcal P}(k)$ contained in $(x_-,x_+)$. The complement mod.$0$ in $(x_-,x_+)$ of their union    is the union 
of two intervals $(x_+(k),x_+)$ and $(x_-,x_-(k))$, which may be degenerate. 

We define $x_+(\ell)$, for $\ell >k$, as the largest endpoint $<x_+$ of an interval of ${\mathcal P}(\ell)$.
One has $x_+(\ell) \ge x_+(\ell -1)$. The interval $(x_+(\ell -1), x_+(\ell))$ is the union mod.$0$ of intervals 
$J_+^{(\ell)}(1),\ldots J_+^{(\ell)}(b_+(\ell))$ of ${\mathcal P}(\ell)$ for some $b_+(\ell)\ge 0$. 

One has 
$\lim_{\ell\rightarrow +\infty}x_+(\ell)=x_+$. 

One defines similarly $x_-(\ell)$ for $\ell >k$ and $b_-(\ell)\ge 0$ with $\lim_{\ell\rightarrow +\infty}x_-(\ell)=x_-$. 

\smallskip
The decomposition of  $(x_-,x_+)$ is 
\begin{equation}\label{eqspacedec}
(x_-,x_+) = \bigsqcup_{1\le b\le b(k)}J^{(k)}(b) \bigsqcup_{\ell >k}\bigsqcup_{\varepsilon = \pm}\bigsqcup_{1\le b\le b_\varepsilon (\ell)} J_\varepsilon^{(\ell)}(b) \;\;\; \hbox{mod.}\, 0\, . 
\end{equation}

The number of intervals of ${\mathcal P}(\ell +1)$ contained in a single interval of ${\mathcal P}(\ell)$ is at most
$$
max_\beta \sum_\alpha B_{\alpha\beta}(n_\ell , n_{\ell+1})\; . 
$$
One has therefore 
\begin{eqnarray*}
b(k) &\le& \Vert{}^tB(n_{k-1},n_k)\Vert \\
b_\varepsilon (\ell) &\le& \Vert{}^tB(n_{\ell-1},n_\ell)\Vert \, , \\
\end{eqnarray*}
for $\varepsilon = \pm$, $\ell >k$.

%%%%%%%%%%%%%%%%%%%%%%%%%%%%%%%%%
\subsection{Special H\"older estimate}\label{ssspecialHolder}
%%%%%%%%%%%%%%%%%%%%%%%%%%%%%%%%

In this section and the next one, $T$ is an i.e.m of restricted Roth type,  $r$ is a real number $>1$ and  
$\varphi$ is a function in $ C^r_0 ({\bf u}(T))$. Substracting $L_1(\varphi)$ if necessary, 
 we  assume that $L_1(\varphi) =0$. From  Lemma \ref{lemHolder3}, one has the estimate
 $$ || S(0,n_k) \varphi ||_{C^0} \leq C || B(0,n_k) ||^{-\delta_1} || \varphi ||_{C^r}. $$
We introduce this estimate into the time decomposition formula (\ref{eqtimedec}), 
using also condition (a) of section \ref{ssRoth} in (\ref{eqbl}), to obtain that the Birkhoff sum of $\varphi$ of any order are uniformly bounded by $C\,|| \varphi ||_{C^r} $.

\smallskip
It follows from the Gottschalk-Hedlund theorem applied to a continuous model of $T$
(cf. \cite{MMY1} p.838-839) that there exists a bounded function $\psi$ on $I(T)$ such that 
$$ \varphi = \psi \circ T -\psi.$$
In \cite{MMY2} p.1598-1599, it is proven that $\psi$ is continuous on 
$[u_0(T), u_d(T)]$. We normalize $\psi$ by $\int_{u_0(T)}^{u_d(T)} \psi(x) dx =0$.
One has 
$$ \Vert \psi \Vert _{C_0} \leq C \Vert \varphi \Vert _{C^r}.$$

\medskip

For an interval $J= (a,b)$, write $\Delta(J) := \psi(b) - \psi(a)$.

For $k \geq 0$, we define a vector $V(k) \in \RA$ by

$$V_{\alpha}(k):= \Delta(\iat(T(n_k))),  \qquad \alpha\in \A$$
\smallskip

The  vectors $V(k)$ are inductively related.

\begin{lemma}\label{lemHolder4}
For $k\ge 0$, one has 
$$|| V(k+1) - (^tB(n_{k},n_{k+1}))^{-1} V(k)|| \leq C\, || B(0,n_{k}) ||^{-\delta_1 /2} || \varphi ||_{C^r} .$$ 
\end{lemma}

\begin{proof}
Let $\alpha \in \A$. One has
\begin{equation}\label{eqHolder5}
 V_{\alpha}(k) =  \sum_{\beta \in \A}\sum_{j \in N(\beta,\alpha))} \Delta  (T(n_k)^j (I_{\beta}^t (T(n_{k+1})))), 
 \end{equation}
with
$$ N(\beta,\alpha) = \{ j \in [0, B_{\beta}(n_k,n_{k+1}), \; T(n_k)^j (I_{\beta}^t (T(n_{k+1}))) \subset \iat (T(n_k)) \}.$$
Recall from section \ref{ssKZ}    that $B_{\beta}(n_k,n_{k+1})$ is the return time of $I_{\beta}^t (T(n_{k+1}))$ into $I(T(n_{k+1}))$ under $T(n_k)$. The cardinality of $ N(\beta,\alpha) $ is equal to $B_{\beta \alpha}(n_k,n_{k+1})$.

For $0 \leq j < B_{\beta}(n_k,n_{k+1})$, we compare $\Delta ( T(n_k)^{j+1} (I_{\beta}^t (T(n_{k+1}))))$ and \\$\Delta(T(n_k)^j (I_{\beta}^t(T(n_{k+1})))) $. Writing 
$T(n_k)^j(I_{\beta}^t(T(n_{k+1}))) = (a,b)$, we have, as $\varphi = \psi \circ T - \psi$
\begin{eqnarray*}
\Delta ( T(n_k)^{j+1} (I_{\beta}^t (T(n_{k+1})))) &-& \Delta(T(n_k)^j (I_{\beta}^t(T(n_{k+1}))))  \\
&=&( \psi(T(n_k) b) - \psi(b)) - (\psi(T(n_k)a) - \psi(a))\\
         &=& S(0,n_k) \varphi(b) -S(0,n_k) \varphi(a).
\end{eqnarray*}
We use Lemma \ref{lemHolder3} and sum over $j$ to obtain, for $0 \leq j < B_{\beta}(n_k,n_{k+1})$

$$  |\Delta ( T(n_k)^j (I_{\beta}^t (T(n_{k+1}))))- V_{\beta}(k+1) | \leq C B_{\beta}(n_k,n_{k+1}) ||B(0,n_k)||^{-\delta_1} ||\varphi||_{C^r}. $$

Bringing this estimate into (\ref{eqHolder5}) gives, as $N(\beta,\alpha)$ has $B_{\beta \, \alpha}(n_k,n_{k+1})$ elements

$$|| V(k) - ^tB(n_k,n_{k+1}) V(k+1) || \leq C ||B(n_k,n_{k+1}) ||^2\; ||B(0,n_k)||^{-\delta_1} ||\varphi||_{C^r}. $$

Taking condition (a) of section \ref{ssRoth} into account, this gives the estimate of the lemma.
\end{proof}

\begin{lemma}\label{lemHolder5}
If a sequence of vectors $(W(k))_{k\geq 0}$ in $\RA$ satisfies for $k \geq 0$
$$|| W(k+1) - (^tB(n_k,n_{k+1}))^{-1} W(k)|| \leq C\, || B(0,n_k) ||^{-\delta_1 /2} || \varphi ||_{C^r} ,$$
and $\lim_{k \to \infty} W(k) = 0$, then one has 
$$ || W(k) || \leq C' || B(0,n_k) ||^{-\delta_2} || \,\varphi ||_{C^r}$$
where  the constant $\delta_2$ depends only on $r$, $\theta$ and $\sigma$.  
\end{lemma}

\begin{proof}
For $k \ge 0$, define $\bar W(k) := \Omega (\pi(n_k)) W(k)$, where $\Omega (\pi(n_k))$ is the antisymmetric matrix associated to the combinatorial data of $T(n_k)$.
From the hypothesis of the lemma and equation (\ref{eqsymp}) in section \ref{ssKZsymplectic}, we have

\begin{equation}\label{eqHolder7}
 || \bar W(k+1) - B(n_k,n_{k+1}) \bar W(k) || \leq C\, || B(0,n_k) ||^{-\delta_1 /2} || \varphi ||_{C^r}.
\end{equation}

We will first deal with the size of $\bar W(k) \mod \Gamma_s (T(n_k))$, then with the size of $\bar W(k)$, and finally with the size of $W(k)$ itself.

\medskip

Define $W_{\flat}(k) := \bar W(k) \mod \Gamma_s (T(n_k)) \in \Gamma_{\partial}(T(n_k))/ \Gamma_s (T(n_k))$ and
 $$r_{\flat} (k+1) := W_{\flat}(k+1) - B_{\flat}(n_k,n_{k+1}) W_{\flat}(k).$$ 
 We have 
$$ || r_{\flat} (k+1)  || \leq C\, || B(0,n_k) ||^{-\delta_1 /2} || \varphi ||_{C^r}.$$

\smallskip

For $k <\ell$, we write $W_{\flat}(\ell) = B_{\flat}(n_{k},n_\ell) (W_{\flat} (k) + R_{\flat}(k,\ell))$ with 

$$ R_{\flat}(k,\ell) = \sum_{k <j \leq \ell} (B_{\flat}(n_{k},n_{j}))^{-1} r_{\flat} (j).$$

We control the norm of $(B_{\flat}(n_k,n_j))^{-1}$  from condition (c) of section \ref{ssRoth} and get

$$ || R_{\flat}(k,\ell) || \leq C || B(0,n_k) ||^{-\delta_1/3}  || \varphi ||_{C^r}.$$

We claim that we have 

\begin{equation}\label{eqHolder8}
 || W_{\flat} (k) || \leq 2 C || B(0,n_k) ||^{-\delta_1/3}  || \varphi ||_{C^r}.
 \end{equation}

 Otherwise, the norm of the norm of the vector $w(k,\ell):= W_{\flat} (k) + R_{\flat}(k,\ell)$ would be bounded from below by $C \,|| B(0,n_k) ||^{-\delta_1/3}  || \varphi ||_{C^r}$. But the nondegenerate symplectic pairing between the lagrangian subspace $\Gamma_s$ and the quotient $\Gamma_{\partial}/ \Gamma_s$ (cf. section \ref{ssKZsymplectic}) gives, for any  vector $v \in \Gamma_s(T(n_k))$
$$ <v,w(k,\ell)> \,=\, <v, (B_{\flat}(n_k,n_\ell))^{-1} W_{\flat} (\ell)>\, =\, < B_{\flat}(n_k,n_\ell)v,W_{\flat} (\ell)>.$$
As $v$ belongs to the stable subspace and $W_{\flat} (\ell)$ converges to $0$, the right-hand side converges to $0$. By the nondegeneracy of the pairing, $w(k,\ell)$ converges as $\ell$ goes to $+\infty$ to $0$  mod. $\Gamma_s (T(n_k))$, a contradiction. The claim is proved.

 We may therefore write $\bar W(k) = W_u(k) + W_s(k)$, with 
 $W_s(k) \in \Gamma_s(T(n_k))$ and $|| W_u (k) || \leq 2\,C || B(0,n_k) ||^{-\delta_1/3}  || \varphi ||_{C^r}$. Define 
 $$r_s(k+1) := W_s(k+1) - B(n_k, n_{k+1}) W_s(k).$$
From (\ref{eqHolder7}), one has,  
 \begin{eqnarray*}
 || r_s(k+1) || &\leq& || W_u(k+1) || + || B(n_k,n_{k+1}) W_u(k)|| + C || B(0,n_k) ||^{-\delta_1/2}  || \varphi ||_{C^r} \\
          & \leq & C || \varphi ||_{C^r} (   || B(0,n_{k+1}) ||^{-\delta_1/3}+  || B(0,n_k) ||^{-\delta_1/3} ||B(n_k,n_{k+1}) ||)\\
          & \leq & C  || B(0,n_k) ||^{-\delta_1/4} || \varphi ||_{C^r}.
 \end{eqnarray*}
 
 The vectors $W_s(k)$ satisfy
 
 $$W_s(k) = B(0,n_k) W_s(0) + \sum_{\ell=1}^k B(n_{\ell},n_k) r_s(\ell).$$

 We estimate the sum in the same way than in the last lines of the proof of Lemma \ref{lemHolder3} to get 
 
 $$ || W_s(k) || \leq C \, ||B(0,n_k) ||^{-\sigma \delta_1/12} || \varphi ||_{C^r}.$$
 
 We now have proven that
 
 \begin{equation}\label{eqHolder9}
 || \bar W(k) || \leq C \, ||B(0,n_k) ||^{-\sigma \delta_1/12} || \varphi ||_{C^r}.
 \end{equation}
 
  We may therefore write $W(k) = W_0(k) + W_*(k)$, with $W_0(k) \in {\rm Ker} \,\Omega (\pi(n_k))$ and $ || W_*(k) || \leq C \, ||B(0,n_k) ||^{-\sigma \delta_1/12} || \varphi ||_{C^r}$.
 Define 
 $$r_0(k+1) := W_0(k+1) - \,^tB(n_k,n_{k+1})^{-1} \,W_0 (k).$$
 
  The hypothesis of the lemma gives the estimate
 
 \begin{eqnarray*}
 || r_0(k+1) || &\leq& || W_*(k+1) || + || ^tB(n_k,n_{k+1})^{-1} W_* (k)|| + C\, || B(0,n_k) ||^{-\delta_1 /2} || \varphi ||_{C^r} \\
           &\leq & C\,  ( ||B(0,n_{k+1}) ||^{-\sigma \delta_1/12} + || ^tB(n_k,n_{k+1})^{-1} || \,  ||B(0,n_k) ||^{-\sigma \delta_1/12})  || \varphi ||_{C^r}  \\
           & \leq & C\,  ||B(0,n_k) ||^{-\sigma \delta_1/16} || \varphi ||_{C^r} .
  \end{eqnarray*}
 
 As the extended Kontsevich-Zorich cocycle acts trivially on the kernels of the
  antisymmetric matrices $\Omega(\pi)$ (cf section \ref{ssKZsymplectic}) and $\lim_{\ell \to \infty}  W_0(\ell) = 0$, we obtain
  
 \begin{eqnarray*}
 || W_0(k) || & \leq & \sum_{\ell>k} || r_0(\ell) || \\
             & \leq &  C\,  ||B(0,n_k) ||^{-\sigma \delta_1/16} || \varphi ||_{C^r} .
   \end{eqnarray*}
   This is the estimate of the lemma, with $\delta_2 = \sigma \delta_1 / 16$.
\end{proof}

By the continuity of $\psi$ and Lemma \ref{lemHolder4}, the hypotheses of Lemma \ref{lemHolder5} are satisfied by the sequence $V(k)$ of Lemma \ref{lemHolder4}. We obtain, for all $\alpha \in \A$

\begin{equation}\label{eqHolderspecial}
\Vert \Delta (\iat(T(n_k))) \Vert \leq C || B(0,n_k) ||^{-\delta_2} || \,\varphi ||_{C^r}.
\end{equation}

%%%%%%%%%%%%%%%%%%%%%%%%%%%%%%%%%%
\subsection{General H\"older estimate}
%%%%%%%%%%%%%%%%%%%%%%%%%%%%

The first step is to extend inequality (\ref{eqHolderspecial}).

\begin{lemma}\label{lemHolder6}
For $k\ge 0$, $\alpha \in \A$, $0 \le N < B_{\alpha}(0,n_k)$, one has

$$ \Vert \Delta (T^N (\iat(T(n_k)))) \Vert \leq C || B(0,n_k) ||^{-\delta_2} \, || \varphi ||_{C^r}.$$

\end{lemma}

\begin{proof}
From $\varphi = \psi \circ T - \psi$, one obtains

\begin{equation}\label{eqgenHolder1}
 \Delta (T^N (\iat(T(n_k))))  -  \Delta (\iat(T(n_k))) = \int_{\iat(T(n_k))}  \sum_{i=0}^{N -1} D\varphi \circ T^i (x) \,dx.
 \end{equation}
 
 Observe that, for $x \in \iat(T(n_k))$, the points $T^i(x), 0 < i <  B_{\alpha}(0,n_k)$ are farther from $0$ than $x$. According to (\ref{eqtimedec}), we write
 
$$ \sum_{i=0}^{N -1} D\varphi \circ T^i (x) =  \sum_{\ell =0}^{k-1} \sum_{0\le b <b(\ell)} S(0,n_\ell )\,D\varphi ((T(n_\ell))^b(y(\ell)))\; .$$

Here, the integers $b(\ell)$ satisfy, according to (\ref{eqbl})

$$b(\ell) \le \Vert B(n_\ell, n_{\ell +1})\Vert \; . $$

We use Lemma \ref{lemHolder2} to estimate the special Birkhoff sums of $D\varphi$.
This gives

$$ \vert \sum_{i=0}^{N -1} D\varphi \circ T^i (x) \vert \leq C  || D\varphi ||_{C^{\rho}}  \sum_{\ell =0}^{k-1} \Vert B(n_\ell, n_{\ell +1})\Vert\; \Vert B(0, n_{\ell})\Vert ^{1-\delta}.$$
From Lemma \ref{lemHolder1} and condition (a), one obtains, as $\delta_2 < \delta$

$$ \vert  \int_{\iat(T(n_k))}  \sum_{i=0}^{N -1} D\varphi \circ T^i (x) \,dx \vert \leq 
C  || D\varphi ||_{C^{\rho}}   \Vert B(0,n_k) \Vert ^{-\delta_2} .$$

The estimate of the lemma now follows from (\ref{eqHolderspecial}) and (\ref{eqgenHolder1}).

\end{proof}

We are now ready to complete the proof of the theorem. Let $x_-<x_+$ be distinct points in 
$I(T)$. Let $k$ be the smallest integer such that $x_-,x_+$  contains 
at least one element of the partition $\mathcal P(k)$ of section \ref{ssspace}. According to (\ref{eqspacedec}), we write
$$(x_-,x_+) = \bigsqcup_{1\le b\le b(k)}J^{(k)}(b) \bigsqcup_{\ell >k}\bigsqcup_{\varepsilon = \pm}\bigsqcup_{1\le b\le b_\varepsilon (\ell)} J_\varepsilon^{(\ell)}(b) \;\;\; \hbox{mod.}\, 0\, . 
$$
Here, $J^{(k)}(b)$ is an element of $\mathcal P(k)$ and $J_\varepsilon^{(\ell)}(b)$ 
is an element of $\mathcal P(\ell)$. One has from section \ref{ssspace}

\begin{eqnarray*}
b(k) &\le& \Vert{}^tB(n_{k-1},n_k)\Vert ,\\
b_\varepsilon (\ell) &\le& \Vert{}^tB(n_{\ell-1},n_\ell)\Vert \, . \\
\end{eqnarray*}

One has 
$$ \psi(x_+) - \psi (x_-) =   \sum_{1\le b\le b(k)}\Delta(J^{(k)}(b)) \sum_{\ell >k}\sum_{\varepsilon = \pm}\sum_{1\le b\le b_\varepsilon (\ell)} \Delta(J_\varepsilon^{(\ell)}(b)).$$

Using Lemma \ref{lemHolder6}, one obtains

\begin{align*}
\vert \psi(x_+) - \psi (x_-) \vert & \leq \,C  \, || \varphi ||_{C^r} \, \sum_{\ell \ge k}   \Vert{}^tB(n_{\ell-1},n_\ell)\Vert \,  || B(0,n_\ell) ||^{-\delta_2} \\
     & \leq \, C  \, || \varphi ||_{C^r} \, || B(0,n_k) ||^{-\delta_2 /2},
\end{align*}

where we have used condition (a) and the fact that $ \Vert{}^tB(n_{\ell-1},n_\ell)\Vert $ and $ \Vert B(n_{\ell-1},n_\ell)\Vert $ have the same order. On the other hand, by the definition of $k$ and Lemma \ref{lemHolder1}, one has, for any $\tau >0$

\begin{align*}
\vert x_+ - x_- \vert & \geq \min _{\alpha \in \A} \vert I_{\alpha}^t (n_k) \vert \\
    & \geq  C_{\tau}^{-1} ||  B(0,n_k) ||^{-\tau}   \max _{\alpha \in \A} \vert I_{\alpha}^t (n_k) \vert \\
    & \geq C_{\tau}^{'-1} \vert I(T) \vert \Vert  B(0,n_k) \Vert ^{-1-\tau}.
\end{align*}

We thus obtain

$$ \vert \psi(x_+) - \psi (x_-) \vert \leq \, C  \, || \varphi ||_{C^r} \, \left ( \frac {\vert x_+ - x_- \vert}{\vert I(T) \vert} \right ) ^{\delta_2/3}.$$

The proof of Theorem \ref{thmHolder1} is complete.

%%%%%%%%%%%%%%%%%%%%%%%%%%%%%
%%%%%%%%%%%%%%%%%%%%%%%%%%%%%%
\subsection{Higher differentiability}\label{sechigh}
%%%%%%%%%%%%%%%%%%%%%%%%%%%%%%%%
%%%%%%%%%%%%%%%%%%%%%%%%%%%%%

\smallskip
To formulate a result allowing for smoother solutions of the cohomological equation, we introduce the same finite-dimensional spaces than in \cite{MMY2} (although the notations are slightly different). In the following definitions, $T$ is an i.e.m., $D$ is a nonnegative integer, and  $r$ is a real number $\geq 0$.

\begin{definition}\label{defcrD} We denote by 

\begin{itemize}
\item $C^{D+r}_{D}({\bf u}(T))$ the subspace of 
$C^{D+r}({\bf u}(T))$ consisting of functions $\varphi$ which satisfy $\partial D^i \varphi =0$ for $0 \leq i \leq D$.
\item  $\Gamma(D,T)$ the subspace of $C^{\infty}_{D}({\bf u}(T))$ consisting  of functions $\chi$ whose restriction to each $[u_{i-1}(T), u_i(T)]$ is a polynomial of degree $\leq D$.
\item  $\Gamma_s (D, T)$ the subspace of $\Gamma(D,T)$ consisting of functions $\chi$ which can be written as $\chi = \psi \circ T - \psi$, for some function $\psi$ of class $C^{D}$ on the closure of $[u_0(T), u_d(T)]$.
\end{itemize}
\end{definition} 

It is proven in \cite{MMY2}  p.1602  that the dimension of $\Gamma (D,T)$ is equal to
$2g + D(2g-1)$, and that the dimension of $\Gamma_s (D,T)$ is equal to $g+D$ when $T$ is of restricted Roth type.

\begin{theorem}\label{thmHolder2}
Let $T$ be an i.e.m.  of restricted Roth type. Let 
$D$  be a nonnegative integer, and let $r$ be a real number $> 1$. Let 
$\Gamma_u(D,T)$ be a subspace of $\Gamma (D,T)$ 
 supplementing $\Gamma_s (D,T)$.  There exist a real number $\bar \delta >0$, and bounded linear operators $L_0: \varphi \mapsto \psi$ from 
 $C^{D+r}_{D} ({\bf u} (T))$ to $C^{D+ \bar \delta} ([u_0(T),u_d(T)])$ and $L_1: \varphi \mapsto \chi$ from 
$C^{D+r}_{D} ({\bf u} (T))$ to $\Gamma_u(D,T)$ such that any $\varphi \in C^{D+r}_{D} ({\bf u} (T))$ satisfies
$$ \varphi = \chi + \psi \circ T - \psi.$$
The number $\bar \delta$ is the same than in Theorem \ref{thmHolder1}.
\end{theorem}

The derivation of this theorem from  Theorem \ref{thmHolder1} (the case $D = 0$) is easy and done in \cite{MMY2} p.1602-1603.

\end{document}